\documentclass {siamltex}

\usepackage[english]{babel}
\usepackage{amsfonts}
\usepackage{amsmath}
\usepackage{amssymb,amsbsy}
\usepackage{amscd}
\usepackage{latexsym}
\usepackage{float}
\usepackage{graphicx}
\usepackage{setspace}
\usepackage{fancyhdr}
\usepackage{color}
\usepackage{stmaryrd,algorithm}

\newcommand{\jump}[1]{\llbracket #1 \rrbracket }

\usepackage{subfigure}

\usepackage{color}
\definecolor{myorange}{rgb}{0.9568,0.4941,0.1961}
\definecolor{myred}{rgb}{0.9098,0.1294,0.2078}
\definecolor{myblue}{rgb}{0.0352,0.4981,0.6509}
\definecolor{myhyperblue}{rgb}{0.1607,0.3922,0.9}
\definecolor{mygreen}{rgb}{0.2235,0.6353,0.2588}
\definecolor{mygrey}{rgb}{0.3,0.3,0.3}

\graphicspath{{Figures/}}
\title{The penalty free Nitsche method and nonconforming finite
  elements for the Signorini problem
}

\author{Erik Burman\thanks{Department of Mathematics, 
University College London, Gower Street, London, 
UK--WC1E  6BT, 
United Kingdom; ({\tt e.burman@ucl.ac.uk})}
\and Peter Hansbo\thanks{Department of Mechanical Engineering, J\"onk\"oping University,
SE-55111 J\"onk\"oping, Sweden; ({\tt peter.hansbo@ju.se})}
\and Mats G. Larson \thanks{Department of Mathematics and Mathematical Statistics, Ume{\aa} University, 
SE-901 87 Ume{\aa}, Sweden; ({\tt mats.larson@math.umu.se})}
}
 
\begin{document}

\maketitle

\begin{abstract}
We design and analyse a Nitsche method for contact problems. Compared
to the seminal work of Chouly and Hild \cite{CH13b} ({\emph{A
    Nitsche-based method for unilateral contact problems: numerical
    analysis}}. SIAM J. Numer. Anal. 51 (2013), no. 2) our method is constructed
by expressing the contact conditions in a nonlinear function for the
displacement variable instead of the lateral forces. The contact
condition is then imposed using the nonsymmetric variant of Nitsche's
method that does not require a penalty term for
stability. Nonconforming piecewise affine 
elements are considered for the bulk discretization. We prove optimal
error estimates in the energy norm.
\end{abstract}



\section{Introduction}
We consider the Signorini problem, find $u$ such that
\begin{equation}\label{signorini}
\begin{array}{rcl}
-\Delta u &=& f \mbox{ in } \Omega \\
u &= & 0  \mbox{ on } \Gamma_D \\
\partial_n u &=& 0 \mbox{ on } \Gamma_N \\
u \leq 0,\; \partial_n u \leq 0, \; u \partial_n u &=& 0 \mbox{ on } \Gamma_C,
\end{array}
\end{equation}
where $f \in L^2(\Omega)$ and $\Omega \subset \mathbb{R}^d$, $d=2,3$
is a convex polygonal (polyhedral) domain with boundary $\partial \Omega$
and $\Gamma_D \cup \Gamma_N \cup \Gamma_C = \partial \Omega$. We assume that
$\Gamma_C$ coincides with one of the sides of the polygon. We write
$\partial_n u := n \cdot \nabla u$, where $n$ denotes the outwards
pointing normal of $\partial \Omega$. 

It is well known that this problem admits a unique solution $u \in
H^1(\Omega)$. This follows from the theory of Stampacchia applied to
the corresponding variational inequality (see for instance \cite{HHN96}). We will also assume the
additional regularity $u \in H^{\frac32+\nu}(\Omega)$ $0 < \nu \leq
\frac12$. There exists a large body of litterature treating finite
element methods for contact
problems.  In
general however, it has proven difficult to prove optimal error
estimates without making assumptions on the regularity of
the exact solution or the constact zone. In the pioneering work of Scarpini and Vivaldi
\cite{SV77} $O(h^{\frac34})$ convergence was proved in the energy norm
for solutions in $H^2(\Omega)$. Brezzi, Hager and Raviart
\cite{BHR77} then proved $O(h)$ convergence under the additional condition
that the solution was in $W^{1,\infty}(\Omega)$ or that the number of
  points where the contact condition changes from binding to
  non-binding is finite. These initial works were
followed by a series of papers where the scope was widened and sharper
estimates obtained
\cite{KO88,GleT89,BB00,BBB03,BBR03,Wriggers2007,WPGW12,CH13a}. Discretization of \eqref{signorini} is usually performed on
the variational inequality or using a penalty method. The first case however leads to some nontrivial
choices in the construction of the discretization spaces in order to
satisfy the nonpenetration condition and it has proved difficult to
obtain optimal error estimates \cite{HR12}. The latter case, leads to the usual
consistency and conditioning issues of penalty methods. A detailed analysis
the penalty method was recently performed by Chouly and Hild \cite{CH13a}. Another approach proposed by Hild and
Renard \cite{HR10} is to use a
stabilized Lagrange-multiplier in the spirit of Barbosa and Hughes \cite{BH92},
using the reformulation of the contact condition
\begin{equation}\label{cond1}
\partial_n u = - \gamma^{-1} [u - \gamma \partial_n u ]_+
\end{equation}
where $[x]_\pm = \pm \max(0,\pm x)$, proposed by Alart and Curnier \cite{AC91} in an augmented Lagrangian
framework.
Using the close relationship between the Barbosa and Hughes method and
Nitsche's method \cite{Nit71} as discussed by Stenberg \cite{Sten95}, this method was
then further developed in the elegant Nitsche type formulation
introduced by Chouly, Hild and Renard 
\cite{CH13b,CHR15}. In these works
optimal error estimates for solutions in $H^{\frac32+\nu}(\Omega)$ to
the above model problem were obtained for
the first time. Their method was proposed in a nonsymmetric and a
symmetric version similar to Nitsche's method for the imposition of
boundary conditions; it has however been observed that in their framework,
there was no equivalent to the penalty--free non--symmetric Nitsche
method proposed in \cite{Bu12}. Our aim in this work is to fill this
gap, rather adding a piece to the puzzle than pretending to
propose a method superior to the previous variants.

The penalty free Nitsche method can be interpreted as a Lagrange
multiplier method where the multiplier and the corresponding
test function has been replaced by the normal flux of the solution
variable and of its test function, respectively. To design this method
for contact problems we take a slightly different approach than in \cite{CH13b}.
Instead of working on the formulation \eqref{cond1} for the
lateral forces we use a similar relation on the displacement:
\begin{equation}\label{cond2}
u = - \gamma [\partial_n u  - \gamma^{-1} u]_+.
\end{equation}
Setting $P_{\gamma}(u) = \gamma \partial_n u - u$
we may write this relation as
\begin{equation}\label{eq:condonu}
u = - [P_{\gamma}(u)]_+.
\end{equation}
It is straightforward to show that this is equivalent to the contact
condition of equation \eqref{signorini}. First assume that
\eqref{cond2} holds. Then by construction $u \leq 0$. For $u=0$ we see
that $[\partial_n u]_+ = 0$ so in this case $\partial_n u \leq 0$. On
the other hand if $u \ne 0$ and 
$\partial_n u > 0$ then $u = - \gamma (\partial_n u  -
\gamma^{-1} u) < u$, which is a contradiction. Similarly if 
$\partial_n u < 0$ and $u \ne 0$ then $u = - \gamma (\partial_n u  -
\gamma^{-1} u)> u$. On the other hand if $u \leq 0$ and $u \partial_n u
 = 0$ then \eqref{cond2} holds and similarly if $\partial_n u \leq 0$ and $u\partial_n u
 = 0$ then \eqref{cond2} holds.

We multiply \eqref{signorini}
by a function $v$ with zero trace on $\Gamma_D$ and apply Green's
formula to obtain
\[
a(u,v)- \left< \partial_n u ,
  v\right>_{\Gamma_C} = (f,v)_{\Omega},
\]
where $(\cdot,\cdot)_{\Omega}$ and $\left< \cdot,\cdot\right>_{\Gamma_C}$ denote the $L^2$-scalar product on
$\Omega$ and $\Gamma_C$ respectively and
$a(u,v) :=(\nabla u,\nabla v)_{\Omega}$.
We then add  a term imposing \eqref{cond1} on the following form
\begin{equation}\label{full_coupling_form}
 \left<u  + \gamma [\partial_n u  - \gamma^{-1} u]_+, \theta_1 \partial_n v
   + \theta_2 \gamma^{-1} v\right>_{\Gamma_C},
\end{equation}
resulting in family of Nitsche formulations defined by two parameters
$\theta_1$ and $\theta_2$,
\[\begin{array}{c}
a(u,v)- \left< \partial_n u ,
  v\right>_{\Gamma_C} +  \theta_1\left< \partial_n
  v,u\right>_{\Gamma_C} + \theta_2\gamma^{-1}
\left<u,v\right>_{\Gamma_C}  \\[3mm]
+ \left<\gamma [\partial_n u  - \gamma^{-1} u]_+, \theta_1 \partial_n v
   + \theta_2 \gamma^{-1} v\right>_{\Gamma_C}= (f,v)_{\Omega}.
\end{array}\]
Taking $\theta_1 \in
\{-1,0,1\}$ and $\theta_2 =1$ results in methods equivalent to those
proposed in \cite{CHR15} on the form
\begin{equation}\label{formal_Nit_contact}\begin{array}{c}
(\nabla u,\nabla v)_{\Omega} - \left< \partial_n u ,
  v\right>_{\Gamma_C} \pm  \left< 
  u, \partial_n v\right>_{\Gamma_C} + \left< \gamma^{-1} u,
  v\right>_{\Gamma_C}\\[3mm]
+ \left< [P_{\gamma}(u)]_+, \gamma^{-1} v \pm \partial_n v
   \right>_{\Gamma_C} = (f,v)_\Omega
\end{array}
\end{equation}
from which we deduce that the linear part of the formulation coincides
with the classical version of Nitsche's method. It is straightforward
to verify that \eqref{formal_Nit_contact} is equivalent with the
formulation proposed in \cite{CHR15}.

Herein we will consider the method obtained
when $\theta_1=1$ and $\theta_2 = 0$ in which case the term imposing
the contact condition reduces to
\[
 \left<u  + \gamma [\partial_n u  - \gamma^{-1} u]_+, \partial_n v
 \right>_{\Gamma_C} .
\]
Observe that the two terms only differ by
the exclusion of the last term which corresponds to a penalty
and in that sense the latter variant is penalty free.

It follows that the penalty free version leads to the following formal restatement of \eqref{signorini} for
smooth $u$
\begin{equation}\label{formal_var_contact}
(\nabla u,\nabla v)_{\Omega} - \left< \partial_n u ,
  v\right>_{\Gamma_C} +  \left< 
  u, \partial_n v\right>_{\Gamma_C} 
+ \left< [P_{\gamma}(u)]_+, \partial_n v
   \right>_{\Gamma_C} = (f,v)_\Omega.
\end{equation}
Observe that the linear part of the system is equivalent to that
proposed in \cite{Bu12} for Dirichlet boundary conditions, but that
here this is used to enforce the condition \eqref{eq:condonu} on $u$.

For the discretization of \eqref{formal_var_contact} we will use the
Crouzeix--Raviart nonconforming piewewise affine element with midpoint
continuity on element edges (or continuity of averages over faces in
three dimensions). As we shall see below, this element is
advantageous for the formulation proposed, since the necessary
stability results are relatively straightforward to prove. The
nonconforming finite element space has been analysed for the Signorini
problem by Hua and Wang \cite{HW07}. They prove optimal convergence up to a
logarithmic factor for $H^2(\Omega)$ solutions under the assumption
that the number of points where the constraint changes from binding
to nonbinding is finite. In this work we prove the same optimal
results for solutions in $H^{\frac32+\nu}(\Omega)$, $\nu>0$ as those
obtained in \cite{CH13b,CHR15}.
 
To handle the nonconformity error we need to make an additional mild
assumption on the source term: the trace of $f$ must be well
defined in the vicinity of the contact boundary $\Gamma_C$. To make
this precise, let
$$\Omega_{t_C}:= \{x \in \bar \Omega: x= y  - n_y t, \mbox{ where } y\in
\Gamma_C \mbox{ and } 0 \leq t \leq t_C\},
$$ 
where $n_y$ denotes the outward pointing normal on $\Gamma_C$ at the
point $y$. For a fixed $t$, we define 
$$
\partial_t
  \Omega := \{x \in \bar \Omega:  x= y  - n_y t, \mbox{ where } y\in
\Gamma_C \}.
$$
Observe that for any
function $v \in H^s(\Omega_{t_C})$ with $s>\tfrac12$ there holds
\begin{equation}\label{eq:tracebound}
\|v\|_{\Omega_C} \lesssim t_C^{\frac12} \sup_{0\leq t \leq t_C} \|v\|_{\partial_t
  \Omega}.
\end{equation}
We introduce the norm $\|u\|_{L^2_{\infty}(\Omega)} :=
\|u\|_{L^2(\Omega)} +  \sup_{0\leq t \leq t_C} \|v\|_{\partial_t
  \Omega}$ and assume that
\begin{equation}\label{f_assump}
\exists ~ t_C>0\mbox{ such that }
\|f\|_{L^2_{\infty}(\Omega)} < \infty.
\end{equation}
\section{The nonconforming finite element method}
To simplify the analysis below we will work with the nonconforming
finite element space proposed by Crouzeix and Raviart in \cite{CR73}. 
Let $\{\mathcal{T}_h\}_h$ denote a family of shape regular and quasi uniform
tessellations of $\Omega$ into nonoverlapping simplices, such that for
any two different simplices $\kappa$, $\kappa' \in \mathcal{T}_h$, $\kappa \cap
\kappa'$ consists of either the empty set, a common face or edge, or a common
vertex. The diameter of a simplex $\kappa$ will be denoted
$h_{\kappa}$ and the outward pointing normal $n_{\kappa}$. The family
$\{\mathcal{T}_h\}_h$ is indexed by the maximum element size of
$\mathcal{T}_h$, $h :=
\max_{\kappa \in \mathcal{T}_h} h_\kappa$. We denote the set of element faces in $\mathcal{T}_h$ by
$\mathcal{F}$ and let $\mathcal{F}_i$ denote the set of interior faces and $\mathcal{F}_{\Gamma}$ the set of faces in some
$\Gamma \subset \partial \Omega$. We will assume that the mesh is
fitted to the subsets of $\partial \Omega$ representing the boundary
conditions $\Gamma_D$, $\Gamma_N$  and $\Gamma_C$, so that the
boundaries of these subsets coincide with the boundaries of subsets of
element faces. 
To each face
$F$ we associate a unit normal vector, $n_F$. For interior faces its orientation is arbitrary, but fixed. On the boundary
$\partial \Omega$ we identify $n_F$ with the outward pointing normal
of $\partial \Omega$. 
The subscript on the normal is dropped in cases where it follows from the context.

We define the jump over interior faces $F \in
\mathcal{F}_i$ by $$\jump{v}\vert_F:= \lim_{\epsilon \rightarrow
  0^+} (v(x\vert_F- \epsilon n_F) - v(x\vert_F+ \epsilon n_F))$$
and for faces on the boundary, $F \in \partial \Omega$, we let
$\jump{v}\vert_F := v \vert_F$. Similarly we define the average of a function over
an interior
face $F$ by $$\{v \}\vert_F := \tfrac12 \lim_{\epsilon \rightarrow
  0^+} (v(x\vert_F- \epsilon n_F) + v(x\vert_F+ \epsilon n_F))$$ and
for $F$ on the boundary we define $\{v \}\vert_F := v \vert_F$.
The classical nonconforming space of piecewise affine
finite element functions (see \cite{CR73}) then reads
$$
V_h := \{v_h \in L^2(\Omega): \int_{F} \jump{v_h}~
\mbox{d}s = 0,\, \forall F \in \mathcal{F}_i \cup \mathcal{F}_{\Gamma_D}
\mbox{ and } v_h\vert_{\kappa} \in \mathbb{P}_1(\kappa),\,
\forall \kappa \in \mathcal{T}_h \}
$$
where $\mathbb{P}_1(\kappa)$ denotes the set of polynomials of degree less than
or equal to one restricted to the element $\kappa$. 

The finite element method takes the form: find $u_h \in
V_h$ such that
\begin{equation}\label{FEM}
A_h(u_h,v_h) = L(v_h), \quad \forall v_h
\in V_h 
\end{equation}
where $L(v_h) := (f,v_h)_\Omega$ and
\begin{equation}\label{Aform}
A_h(u_h,v_h):= a_h(u_h,v_h) + \left<u_h + [P_\gamma(u_h)]_+, \partial_n v_h \right>_{\Gamma_C} 
\end{equation}
with $P_\gamma(u_h) = \gamma\partial_n u_h -
 u_h$ and $\gamma>0$ a parameter to determine. The
linear form $a_h(\cdot,\cdot)$ coincides with the consistent part of Nitsche's method,
\[
a_h(u_h,v_h) := a(u_h,v_h) - \left< \partial_n u_h ,
  v_h\right>_{\Gamma_C} 
\]
where we have redefined $a(u,v) :=\sum_{\kappa \in \mathcal{T}_h} (\nabla u_h,\nabla v_h)_{K}$.
To see the effect of the nonlinear term, let $\Gamma_C^+$ denote the
part of the contact zone where $\gamma [\partial_n u  - \gamma^{-1}
u]_+>0$ and $\Gamma_C^0 = \Gamma_C \setminus \Gamma_C^+$. We may then
write the form $A(\cdot,\cdot)$
\[
a(u_h,v_h) - \left< \partial_n u ,
  v_h\right>_{\Gamma_C} + \left< \partial_n v,
  u_h\right>_{\Gamma^0_C} +\left<\gamma \partial_n u , \partial_n v \right>_{\Gamma^+_C}.
\]
This corroborates the naive idea that the method should impose a
Dirichlet condition on $\Gamma_C^0$, here using the penalty free
Nitsche method, and a Neumann condition on
$\Gamma_C^+$, here in the form of a penalty term. Observe that the continuity of the form that is obvious
in the formulation \eqref{Aform} (by the continuity of $[\cdot]_+$,
see more details below) is
no longer clear in this latter expression. 

For comparison, in the method of Chouly, Hild and Renard, the form \eqref{formal_Nit_contact}
$A(\cdot,\cdot)$ takes the form
\begin{multline*}
a(u_h,v_h) - \left< \partial_n u ,
  v_h\right>_{\Gamma^0_C} +\theta_1 \left< \partial v_n,
  u_h\right>_{\Gamma^0_C} + \left<u_h ,\gamma^{-1}
  v_h\right>_{\Gamma^0_C} +\theta_1\left<\gamma \partial_n u , \partial_n v \right>_{\Gamma^+_C},
\end{multline*}
where $\theta$ takes the values $-1$ or $1$ for the symmetric and
nonsymmetric versions respectively. Clearly in this case the Dirichlet
condition on $\Gamma^0_C$ is imposed using the classical Nitsche
method and the Neumann condition on $\Gamma_C^+$ is imposed either
weakly or with an additional penalty term (in the symmetric case, this
term has the wrong sign and does not stabilize the boundary condition).
\subsection{Preliminary results}
For the analysis below we will use some elementary tools that we
collect here. We will use the notation $a \lesssim b$ for
$a \leq C b$ where $C$ is a constant independent of $h$.

The following norms on $H^{\frac32+\nu}(\Omega)+V_h$ will be used below
simplify to simplify the notation, 
\[
\|v\|_{h,\Omega} := \left(\sum_{\kappa \in \mathcal{T}_h}
  \|v\|_{\kappa}^2 \right)^{\frac12},\quad \|v\|_{h,\Gamma_C} := \left(\sum_{F \in \mathcal{F}_{\Gamma_C}}
  \|v\|_{F}^2 \right)^{\frac12},
\]
the broken $H^1$-norms,
\[
\|v\|_{1,h}:=\|\nabla v \|_{h,\Omega} + \|v \|_{h,\Omega}
\]
and
\[
\|v\|_{1,C}:=\|v\|_{1,h}+ \gamma^{\frac12}  \|\partial_n 
  v \|_{h,\Gamma_C}+\gamma^{-\frac12} \|v\|_{h,\Gamma_C}.
\]
We recall, for future reference, the following inequalities:
\begin{itemize}
\item Poincar\'e inequality,
there exists $\alpha>0$ such that
\begin{equation}\label{poincare}
\alpha \|v\|^2_{1,h} \leq \|\nabla v\|^2_{h} \quad \forall
v \in V_h + {H^1(\Omega)}.
\end{equation}
\item Inverse 
inequality,
\begin{equation}\label{inverse}
|v|_{H^1(\kappa)} \leq C_I h^{-1}_\kappa \|v\|_{L^2(\kappa)}\quad \forall v \in \mathbb{P}_1(\kappa).
\end{equation}
\item Trace inequalities,
\begin{equation}\label{trace0}
\|v\|_{L^2(\partial \kappa)} \leq C_T \left( h_\kappa^{-\frac12} \|v\|_{L^2(\kappa)} +
h_\kappa^{\frac12} |v|_{H^1(\kappa)}\right)\quad \forall v \in H^1(\kappa)
\end{equation}
and 
\begin{equation}\label{trace}
\|v\|_{L^2(\partial \kappa)} \leq C_t h_\kappa^{-\frac12} \|v\|_{L^2(\kappa)} \quad\forall v \in \mathbb{P}_1(\kappa).
\end{equation}
\end{itemize}
For the analysis below we also need a quasi-interpolation operator
that maps piecewise linear nonconforming functions into the space of
piecewise linear conforming functions. 
Let $I_\text{cf}: V_h \mapsto V_h \cap H^1(\Omega) $ denote a quasi interpolation operator \cite{HW96,ABC03,KP03} such that
\begin{equation}\label{discrete_interp}
\| I_\text{cf} v_h - v_h \|_\Omega+ h \|\nabla(I_\text{cf} v_h -
v_h)\|_{h} \lesssim \|h^{\frac12} \jump{v_h}\|_{\mathcal{F}_i}\lesssim
h \|\nabla v_h\|_h .
\end{equation}
Stability is based on the fact that we can construct a function which
is zero in the bulk of the domain and with a certain value of the flux
on the boundary. We make this precise in the following lemma.
\begin{lemma}\label{specfunc}
Let $r:\Gamma_C \mapsto \mathbb{R}$ be a face--wise constant function such that $r\vert_F
\in \mathbb{R}$ for all $F\in \mathcal{F}_{\Gamma_C}$.
There exists $v_h \in V_h$ such that 
\begin{equation}\label{eq:v1}
\partial_n v_h\vert_F =r_F \in \mathbb{R} \quad \text{ for } F \in \mathcal{F}_{\Gamma_C},
\end{equation}
\begin{equation}\label{eq:v2}
\int_F  \{v_h\} ~\mbox{d}s = 0 \quad\text{for}\; F \in \mathcal{F}_i \cup
\mathcal{F}_{\Gamma_D}\cup
\mathcal{F}_{\Gamma_N}
\end{equation}
and
\begin{equation}\label{funcstab}
\|v_h\|_{\Omega} \lesssim h^{\frac32} \|r\|_{\Gamma_C}. 
\end{equation}
\end{lemma}
\begin{proof}
For a given simplex $\kappa$ with one face in
$\mathcal{F}_{\Gamma_C}$, assume that  $x_1,\hdots,x_d$ are the
vertices in $\Gamma_C$ and $x_0$ is the vertex in the bulk. Define $v_\kappa
\in \mathbb{P}_1(\kappa)$ by $
v_{\kappa}(x_i)=1$,
$i=1,\hdots,d$ and $v_h(x_0)=1-d$. Then it follows that for $F
\subset \partial \kappa \cap \Omega$
\[
\int_{F} v_\kappa ~\mbox{d}x = 0
\]
and  $\nabla v_\kappa:= |\nabla v_h| n_{\partial \Omega}$, where
$n_{\partial \Omega}$ is the normal to $\Omega$ on $\partial \kappa
\cap \partial \Omega$ and $|\nabla v_h|=c_\kappa h_{\kappa}^{-1}$ where $c_\kappa$ is a positive constant that depends only on the shape
regularity of $\kappa$. It follows that
\[
v_h := \sum_{\kappa \in \mathcal{T}_h} v_\kappa \in V_h.
\]
We conclude by multiplying $v_h$ in each
element with $h_{\kappa} c_\kappa^{-1} r_F$. Then by construction
\eqref{eq:v1} and \eqref{eq:v2} are satisfied. The stability
\eqref{funcstab} is a consequence of an inverse trace inequality,
\[
\|v_h\|_{\Omega} \lesssim \left(\sum_{F \in \mathcal{F}_{\Gamma_C}}
  h_\kappa \|h_{\kappa} c_\kappa^{-1} r_F\|_F^2 \right)^{\frac12}
\lesssim h^{\frac32} \|r\|_{\Gamma_C}.
\]
\end{proof}

The nonlinearity satisfies the following monotonicity and continuity
properties.
\begin{lemma}\label{nonlinprop}
Let $a,b \in \mathbb{R}$ then there holds
\[
([a]_+-[b]_+)^2 \leq ([a]_+ - [b]_+) (a - b),
\]
\[
|[a]_+-[b]_+| \leq |a-b|.
\]
\end{lemma}
\begin{proof}
Developing the left hand side of the expression we have
\[
[a]_+^2 +[b]_+^2 - 2 [a]_+[b]_+ \leq  [a]_+ a + [b]_+ b -  a [b]_+
-  [a]_+b =  ([a]_+ - [b]_+) (a - b).
\]
The second claim is trivially true in case both $a$
and $b$ are positive or negative. If $a$ is negative and $b$ positive then
\[
|[a]_+-[b]_+| = |b| \leq |b-a|
\]
and similarly if $b$ is negative and $a$ positive
\[
|[a]_+-[b]_+| = |a| \leq |b-a|.
\]
\end{proof}
\begin{lemma}(Continuity of $A_h$)\label{lem:continuityA}
Let $v_1,v_2 \in H^{\frac32+\nu}+V_h$ and $w_h \in V_h$. Then there holds
\[
|A_h(v_1,w_h) - A_h(v_2,w_h)| \lesssim  \|v_1 - v_2\|_{1,C}
 \|w_h\|_{1,C} \lesssim \Theta(h)^2\|v_1 - v_2\|_{\Omega} \|w_h\|_{\Omega}.
\]
\end{lemma}
\begin{proof}
By the Cauchy--Schwarz inequality we have
\[
a_h(v_1-v_2,w_h)| \leq \|v_1 - v_2\|_{1,C} \|w_h\|_{1,C}.
\]
For the nonlinear term the following bound holds as a consequence
of the third inequality of Lemma \ref{nonlinprop} and the inequalities
\eqref{inverse}--\eqref{trace}:
\[\begin{array}{l}
\left< \gamma ( [\partial_n v_1 - \gamma^{-1} v_1]_+ - [\partial_n v_2 - \gamma^{-1} v_2]_+, \partial_n w_h
   \right>_{\Gamma_C} \\[3mm]
\qquad \leq \left< (|\gamma^{\frac12} \partial_n( v_1-v_2) - \gamma^{-1/2} (v_1-v_2)|,\gamma^{\frac12} |\partial_n w_h
   |\right>_{\Gamma_C} \\[3mm]
   \qquad  \lesssim \|v_1 - v_2\|_{1,C}
 \|w_h\|_{1,C}\\[3mm]\qquad
\lesssim \Theta(h)^2 \|u_1 - u_2\|_{\Omega}
\|w_h\|_{\Omega} 
\end{array}\]
with $\Theta(h):=1+h^{-1}( C_I + C_t C_I \gamma^{\frac12} h^{-\frac12} + C_t
\gamma^{-\frac12} h^{\frac12})$.
\end{proof}
\section{Existence and uniqueness of discrete solutions}
In this section we will prove that the finite dimensional nonlinear
system \eqref{FEM} admits a unique solution under suitable assumptions
on the parameter $\gamma$. First, with $N_V:= \mbox{dim}~ V_h$ define the mapping $G:\mathbb{R}^{N_V} \mapsto \mathbb{R}^{N_V}$ by
\begin{equation}\label{nonlin_syst}
(G(U),V)_{\mathbb{R}^{N_V}} :=A_h(u_h,v_h) - L(v_h),
\end{equation}
where $U = \{u_i\}$, with $u_i$ denoting the degrees of freedom of
$V_h$ associated with the Crouzeix-Raviart basis functions $\{\varphi_i\}_{i=1}^{N_V}$ and similarly
$V=\{v_i\}$ denotes the vector of degrees of freedom associated with the
test function $v_h$. The nonlinear system associated to \eqref{FEM} may
then be written, find $U \in \mathbb{R}^{N_V}$ such that $G(U)=0$.


Let us next prove a positivity result for the formulation \eqref{FEM}
that will be useful when proving existence and uniqueness.
\begin{proposition}\label{positivity}
Assume that $\gamma = \gamma_0 h$ with $\gamma_0$ large enough. Then, for $u_1,u_2\in
V_h$, there exists $v_ h\in
V_h$ such that
\begin{equation}\label{eq:monotone}
\begin{array}{c}\alpha\| u_1 - u_2\|^2_{1,h} + \gamma^{-1} \|u_1-u_2+  [P_\gamma(u_1)]_+- [P_\gamma(u_2)]_+\|^2_{\Gamma_C}\\[3mm]
 \lesssim A_h(u_1,v_h)-A_h(u_2,v_h).
\end{array}\end{equation}
Moreover, for $\gamma_0$ large enough, there exists $B\in
\mathbb{R}^{N_V\times N_V}$ such that for $X$ with $|X|_{\mathbb{R}^{N_V}}$ large enough
\[
(G(X),BX)_{\mathbb{R}^{N_V}} > 0
\]
and there exists $b_1,b_2>0$ associated to $B$ such that
\[
b_1 |X|_{\mathbb{R}^{N_V}} \leq |BX|_{\mathbb{R}^{N_V}} \leq b_2 |X|_{\mathbb{R}^{N_V}}.
\]
\end{proposition}
\begin{proof}
Let $w_h:= u_1-u_2$. Observe that by Lemma \ref{specfunc}
we can choose $x_h(w_h) \in V_h$ such that 
\begin{equation}\label{eq:x1}
\partial_n x_h\vert_F = \gamma^{-1} |F|^{-1} \int_F  w_h ~\mbox{d}s =:
\gamma^{-1} \bar w\vert_F, \quad \mbox{ for } F \in \mathcal{F}_{\Gamma_C}
\end{equation}
and 
\begin{equation}\label{eq:x2}
\int_F  \{x_h\} ~\mbox{d}s = 0 \mbox{ for  } F \in \mathcal{F}_i \cup
\mathcal{F}_{\Gamma_D}\cup
\mathcal{F}_{\Gamma_N}.
\end{equation}
It follows using integration by parts that for all $y_h \in V_h$ there holds
\[
(\nabla y_h,\nabla x_h) - \left<\partial_n y_h , x_h
\right>_{\Gamma_C} = 0.
\]
Now taking $v_h = w_h + x_h$ leads to
\begin{align*}
A_h(u_1,v_h)-A(u_2,v_h)  ={}& \|\nabla w_h\|^2_h  + (\nabla w_h,\nabla x_h)_h - \left<\partial_n w_h, x_h
\right>_{\Gamma_C}  + \left<\gamma^{-1} \bar w, w_h \right> \\  & + \left< [P_\gamma(u_1)]_+- [P_\gamma(u_2)]_+, \partial_n w_h + \gamma^{-1} \bar w
\right>_{\Gamma_C} \\
 = {}&\|\nabla w_h\|^2_\Omega + \left<\gamma^{-1} \bar w, w_h \right>\\ & + \left< [P_\gamma(u_1)]_+- [P_\gamma(u_2)]_+, \partial_n w_h - \gamma^{-1} w_h
\right>_{\Gamma_C} \\
&+ \left< [P_\gamma(u_1)]_+- [P_\gamma(u_2)]_+, \gamma^{-1} (w_h + \bar w) \right>_{\Gamma_C} \\
 = {}&\|\nabla w_h\|^2_\Omega + \left<\gamma^{-1} w_h, w_h \right> \\ & + \left< [P_\gamma(u_1)]_+- [P_\gamma(u_2)]_+, \partial_n w_h - \gamma^{-1} w_h
\right>_{\Gamma_C} \\
& +2  \left< [P_\gamma(u_1)]_+- [P_\gamma(u_2)]_+, \gamma^{-1} w_h \right>_{\Gamma_C} \\
& + \left< [P_\gamma(u_1)]_+- [P_\gamma(u_2)]_+ + w_h, \gamma^{-1} (w_h -\bar w)
\right>_{\Gamma_C}.
\end{align*}
Applying the monotonicity $$\gamma^{-1} \| [P_\gamma(u_1)]_+- [P_\gamma(u_2)]_+\|_{\Gamma_C}^2 \leq  \left< [P_\gamma(u_1)]_+- [P_\gamma(u_2)]_+, \partial_n w_h - \gamma^{-1} w_h
\right>_{\Gamma_C}$$ 
we see that
\[\begin{array}{c}
\gamma^{-1} \| [P_\gamma(u_1)]_+- [P_\gamma(u_2)]_++w_h\|_{\Gamma_C}^2 \leq  \left< [P_\gamma(u_1)]_+- [P_\gamma(u_2)]_+, \partial_n w_h - \gamma^{-1} w_h
\right>_{\Gamma_C}\\[3mm]
 + \left<\gamma^{-1} w_h, w_h \right>_{\Gamma_C} + 2  \left< [P_\gamma(u_1)]_+- [P_\gamma(u_2)]_+, \gamma^{-1} w_h \right>_{\Gamma_C}.
\end{array}\]

Then, using the arithmetic-geometric inequality together
with the approximation properties of the piecewise constant
approximation $\bar w$ and an elementwise trace inequality to get the bound
\[\begin{array}{c}
\left< [P_\gamma(u_1)]_+- [P_\gamma(u_2)]_+ + w_h, \gamma^{-1} (w_h -\bar w)
\right>_{\Gamma_C} \\[3mm] \leq \frac12 \gamma^{-1} \|[P_\gamma(u_1)]_+- [P_\gamma(u_2)]_+ + w_h\|_{\Gamma_C}^2 +
  \frac12 \gamma^{-1} C h \|\nabla w_h\|^2_\Omega
\end{array}\]
we finally obtain
\[\begin{array}{c}
(1 -  \frac12 \gamma^{-1} C h) \|\nabla
w_h\|^2_h +   \frac12 \gamma^{-1}\| [P_\gamma(u_1)]_+- [P_\gamma(u_2)]_+ + w_h\|_{\Gamma_C}^2\\[3mm]  \leq
  A_h(u_1,v_h) -A_h(u_2,v_h).
\end{array}\]
We conclude by choosing $\gamma > C h$.

For the second claim, first consider equation \eqref{eq:monotone} with
$u_1=u_h$, $u_2=0$,
\begin{equation}\label{eq:pos_u}
\alpha\| u_h \|^2_{1,h} + \gamma^{-1} \|u_h-  [P_\gamma(u_h)]_+]_+\|^2_{\Gamma_C} \lesssim A_h(u_h,u_h + x_h(u_h)).
\end{equation}
Let the positive constants $c_h$ and $C_h$ denote the square roots of
the   smallest and the
largest eigenvalues respectively of the matrix given by $(\varphi_i,\varphi_j)_{\Omega}$, $1\leq i,j\leq
N_V$ such that
\[
c_h |U|_{\mathbb{R}^{N_V}} \leq  \|u_h\|_{\Omega} \leq C_h |U|_{\mathbb{R}^{N_V}}.
\]
Let $B$ denote the transformation matrix such that the finite element function
corresponding to the vector $B U$ is the function $u_h +
x_h(u_h)$. First we show that for  $\gamma$ sufficiently large there are constants $b_1$ and $b_2$ such that $b_1 |U|_{\mathbb{R}^{N_V}}
\leq |BU|_{\mathbb{R}^{N_V}} \leq b_2 |U|_{\mathbb{R}^{N_V}}$. This can be seen by observing that
$$
\|u_h\|_\Omega \leq \|u_h+x_h\|_\Omega+\|x_h\|_\Omega\leq 
C_h |B U|_{\mathbb{R}^{N_V}} + C \gamma^{-1} h \|u_h\|_\Omega
$$
so that 
$$
c_h  (1 - C \gamma^{-1} h)|U|_{\mathbb{R}^{N_V}} \leq  (1 - C \gamma^{-1} h) \|u_h\|_\Omega \leq C_h |BU|_{\mathbb{R}^{N_V}}.
$$
Similarly we may prove the upper bound using $c_h |B
U|_{\mathbb{R}^{N_V}} \leq \|u_h+x_h\|_\Omega$ so that
\begin{multline*}
c_h |B U|_{\mathbb{R}^{N_V}} \leq \|u_h\|_{\Omega}+\|x_h\|_{\Omega}
\leq \|u_h\|_{\Omega}+ C \gamma^{-1} h \|u_h\|_\Omega \leq C_h  (1 + C \gamma^{-1} h)|U|_{\mathbb{R}^{N_V}} .
\end{multline*}
Then there holds using 
\eqref{eq:pos_u},
\begin{align*}
(G(U),B U)_{\mathbb{R}^{N_V}}{}& =   A_h(u_h,u_h +
x_h(u_h)) - L(u_h +
x_h(u_h)) \\
{}&\ge 
\frac{\alpha}{2} \|u_h\|^2_{1,h} - \frac{C_*^2}{2 \alpha}
\|f\|^2_{\Omega}
 \ge \frac{\alpha}{2}  \lambda_1 |U|_{\mathbb{R}^{N_V}}^2 - \frac{C_*^2}{2 \alpha} \|f\|^2_{\Omega} .
\end{align*}
where $C_*$ is the constant such that $L(u_h +
x_h(u_h)) \leq C_*
\|f\|_{\Omega}\|u_h\|_{1,h}$ and $\lambda_1$ is the smallest eigenvalue of the matrix defined by
$(\nabla \varphi_i,\nabla \varphi_j)_h+(\varphi_i,\varphi_j)_{\Omega}$, $1\leq i,j\leq
N_V$. We conclude that for 
\[
|U|_{\mathbb{R}^{N_V}} >  \frac{C_*}{\alpha \lambda_1^{\frac12}} \|f\|_{\Omega}
\]
there holds
\[
(G(U),B U)_{\mathbb{R}^{N_V}} >0.
\]
\end{proof}
\begin{proposition}
The formulation \eqref{FEM} admits a unique solution
for $\gamma = \gamma_0 h$, with $\gamma_0$ large enough.
\end{proposition}

\begin{proof}
Fix $h>0$.
Observe that $G$ defined by \eqref{nonlin_syst} is continuous since by Lemma \ref{lem:continuityA}
\begin{align*}
|G(U_1) - G(U_2)|_{\mathbb{R}^{N_V}} = & \sup_{W \in \mathbb{R}^{N_V}: |W|=1} (G(U_1)
-G(U_2),W)_{\mathbb{R}^{N_V}} \\
= & \sup_{w \in V_h} (A_h(u_1,w_h) - A_h(u_2,w_h))  \\
\lesssim & \Theta(h)^2 \|u_1 - u_2\|_{\Omega}
\|w_h\|_{\Omega} \leq
    \Theta(h)^2 C^2_h |U_1 - U_2|_{\mathbb{R}^{N_V}}.
\end{align*}

By the second claim of Proposition \ref{positivity} we may fix $q \in
\mathbb{R}_+$ such that for $X \in \mathbb{R}^{N_V}$ with
$|X| \ge q$ there holds 
\begin{equation}\label{eq:Gpos}
(G(X), B X)_{\mathbb{R}^{N_V}} > 0.
\end{equation}
Assume that there exists no $X \in \mathbb{R}^{N_V}$ such that $G(X)=0$ and define the
function $\phi(X) = -q/{b_1} B^\text{T}
G(X)/|G(X)|_{\mathbb{R}^{N_V}}$. Since $G(X)\ne 0$ and by the
continuity of $G(X)$ $\phi(\cdot)$ is well defined
and continuous. The transpose of
$B$ satisfies the same bounds as $B$ and therefore $\phi$ maps the ball of radius
$q b_2/b_1$ in $\mathbb{R}^{N_V}$ into itself. It then follows by Brouwers fixed point
theorem that $\phi$ admits a fixed point: there exists $Z \in
\mathbb{R}^{N_V}$ with $|Z|_{\mathbb{R}^{N_V}}\ge q$ such that
\[
Z = \phi(Z) =- q/b_1 B^\text{T} G(Z)/|G(Z)|_{\mathbb{R}^{N_V}}.
\]
By definition then $|Z|_{\mathbb{R}^{N_V}}^2= -q/b_1 (G(Z),B Z)_{\mathbb{R}^{N_V}}/|G(Z)|_{\mathbb{R}^{N_V}}<0$ which 
contradicts the assumption \eqref{eq:Gpos}. It follows that there exists at least one
$U\in \mathbb{R}^{N_V}$ such that $G(U)=0$.

Uniqueness of the discrete solution is an immediate consequence of
Proposition \eqref{positivity}. Indeed assume that $u_1$ and $u_2$
both are solutions to \eqref{FEM}, then for $v_h$ chosen as in the Proposition,
\[
\alpha \| u_1 - u_2\|^2_{1,h}
\lesssim A_h(u_1,v_h)-A_h(u_2,v_h) = (f,v_h)_\Omega- (f,v_h)_\Omega = 0.
\]

\end{proof}
\section{A priori error estimates}
A priori error estimates may now be derived by combining the
techniques of the uniqueness argument above with the Galerkin
perturbation arguments. 
\begin{theorem}\label{thm:apriori}
Assume that $u \in H^{\frac32+\nu}(\Omega)$, with
$0<\nu\leq \tfrac12$ is the solution of the problem
\eqref{signorini}.
Assume that  $u_h$ denotes the
solution of \eqref{FEM}-\eqref{Aform}  where $\gamma =
\gamma_0 h$. If $\gamma_0$ is chosen sufficiently large and $h\leq t_C$, where
$t_C$ is the constant of assumption \eqref{f_assump}, then there holds, with $e:=u-u_h$,
\begin{align*}
\alpha^{\frac12} \|e\|_{1,h} + \gamma^{-\frac12}\| [P_\gamma
u_h]_++  u_h\|_{\Gamma_C} 
\lesssim {}& \inf_{v_h \in V_h} (\|u - v_h\|_{1,C} + h^{\frac12}
\|\partial_n (u -v_h)\|_{\mathcal{F}_i}) \\
& + h \|f\|_{L^2_\infty(\Omega)}.
\end{align*}
\end{theorem}
\begin{proof}
Using the definition of the form
$a(\cdot,\cdot)$ we have
\begin{equation}\label{coercivity1}
 \|\nabla e\|_{h}^2 \leq a(e,e)  =
a(e,u-v_h) + a(e,v_h-u_h).
\end{equation}
For the first term we have
\begin{equation}\label{splita1}
a(e,u-v_h) \leq \frac{1}{2} \|\nabla e\|_{h}^2 +
\frac{1}{2} \|\nabla (u - v_h)\|_{h}^2.
\end{equation}
Considering the second term we see that
\begin{align}\nonumber
a(e,v_h-u_h) = {}& \left<\{\partial_n u\},\jump{v_h-u_h}
\right>_{\mathcal{F}\setminus \mathcal{F}_{\Gamma_C}}
+\left<\partial_n e, v_h-u_h\right>_{\Gamma_C} \\ \nonumber
& - \left<\partial_n (v_h-u_h), e\right>_{\Gamma_C} \\ \label{0st_step}
& - \left<[P_\gamma u]_+ - [P_\gamma u_h]_+, \partial_n
  (v_h-u_h) \right>_{\Gamma_C}
\end{align}
Using that
\begin{align*}
\left<\partial_n e, v_h-u_h\right>_{\Gamma_C} 
- \left<\partial_n (v_h-u_h), e\right>_{\Gamma_C} 
= 
\left<\partial_n e, v_h-u\right>_{\Gamma_C}  - \left<\partial_n (v_h-u), e\right>_{\Gamma_C}
\end{align*}
and
\begin{align*}
\left<[P_\gamma u]_+ - [P_\gamma u_h]_+, \partial_n
  (v_h-u_h) \right>_{\Gamma_C} = {}& \left<[P_\gamma u]_+ - [P_\gamma u_h]_+, \partial_n
  (v_h-u) \right>_{\Gamma_C} \\ & + \left<[P_\gamma u]_+ - [P_\gamma u_h]_+, \partial_n
  e \right>_{\Gamma_C}\\
= {}&\left<[P_\gamma u]_+ - [P_\gamma u_h]_+, \partial_n
  (v_h-u) \right>_{\Gamma_C}\\ & + \gamma^{-1} \left<[P_\gamma u]_+ - [P_\gamma u_h]_+, 
  P_\gamma e \right>_{\Gamma_C} \\
& + \gamma^{-1}\left<[P_\gamma u]_+ - [P_\gamma u_h]_+, 
 e \right>_{\Gamma_C}
\end{align*}
we arrive at the identity
\begin{align} \nonumber
a(e,v_h-u_h) = {}&\left<\{\partial_n u\},\jump{v_h-u_h} \right>_{\mathcal{F}\setminus \mathcal{F}_{\Gamma_C}}+
\left<\partial_n e, v_h-u\right>_{\Gamma_C} 
\\ \nonumber & - \left<\partial_n (v_h-u),e+ ([P_\gamma u]_+ - [P_\gamma u_h]_+)\right>_{\Gamma_C} \\ \nonumber &
-  \gamma^{-1} \left<[P_\gamma u]_+ - [P_\gamma u_h]_+, 
  P_\gamma (u-u_h) \right>_{\Gamma_C} \\ 
& - \gamma^{-1}\left<e,[P_\gamma u]_+ - [P_\gamma u_h]_+\right>_{\Gamma_C}.\label{1st_step}
\end{align}
%
Observe now that the following relation holds using monotonicity and the
elementary relation $a^2+b^2+2ab= (a+b)^2$, with $a=\gamma^{-1/2}
(u-u_h)$ and $b=\gamma^{-1/2}([P_\gamma u]_+ - [P_\gamma u_h]_+)$,
\[\begin{array}{c}
- \gamma^{-1} \|e\|_{\Gamma_C}^2 - \gamma^{-1} \left<[P_\gamma u]_+ - [P_\gamma u_h]_+, P_\gamma e\right>_{\Gamma_C} - \gamma^{-1}\left<[P_\gamma u]_+ - [P_\gamma u_h]_+, 
2 e \right>_{\Gamma_C} \\[3mm]
\leq - \gamma^{-1} \|e\|_{\Gamma_C}^2
 - \gamma^{-1} \|[P_\gamma u]_+ - [P_\gamma u_h]_+\|_{\Gamma_C}^2 -
 \gamma^{-1} \left<[P_\gamma u]_+ - [P_\gamma u_h]_+, 
2 e \right>_{\Gamma_C} \\[3mm]
\leq -\|\gamma^{-\frac12}( [P_\gamma u_h]_++ u_h)\|_{\Gamma_C}^2.
\end{array}\]
We deduce the following bound
\begin{align}\nonumber
a(e,v_h-u_h) \leq  {}& \left<\{\partial_n u\},\jump{v_h-u_h} \right>_{\mathcal{F}\setminus \mathcal{F}_{\Gamma_C}}+
\left<\partial_n e, v_h-u\right>_{\Gamma_C} \\ \nonumber &
- \left<\partial_n (v_h-u), e+ ([P_\gamma u]_+ - [P_\gamma u_h]_+)\right>_{\Gamma_C} \\ \nonumber &
  -\|\gamma^{-\frac12}( [P_\gamma u_h]_++ u_h)\|_{\Gamma_C}^2
+  \gamma^{-1}\left<e,[P_\gamma u]_+ - [P_\gamma u_h]_+\right>_{\Gamma_C} \\ & +\gamma^{-1} \|e\|_{\Gamma_C}^2.\label{1st_step}
\end{align}
Choosing now $x_h \in V_h$ as in Lemma \ref{specfunc},
but with $\partial_n x_h\vert_F = \gamma^{-1} (\bar u - \bar
u_h)\vert_F= \gamma^{-1} \bar e\vert_F$ on faces $F \subset \Gamma_C$ we
obtain
\[
a_h(e,x_h) -\left<\{\partial_n u\},\jump{x_h} \right>_{\mathcal{F}}+
\gamma^{-1}\| \bar e\|_{\Gamma_C}^2 
+
\left<[P_\gamma u]_+ - [P_\gamma u_h]_+, \gamma^{-1} \bar e\right>_{\Gamma_C} = 0.
\]
Note that using orthogonality on the faces 
\eqref{eq:condonu} we have
\[
\gamma^{-1}\| \bar e\|_{\Gamma_C}^2 = \gamma^{-1}\| e\|_{\Gamma_C}^2 - \gamma^{-1}\| \bar e -e\|_{\Gamma_C}^2
\]
and once again using orthogonality and also the contact condition
\begin{align*}
\left<[P_\gamma u]_+ - [P_\gamma u_h]_+, \gamma^{-1} \bar e\right>_{\Gamma_C} 
= {}& \left<[P_\gamma u]_+ - [P_\gamma u_h]_++ e, \gamma^{-1} e\right>_{\Gamma_C}\\
& -\left< [P_\gamma u_h]_++ 
  u_h, \gamma^{-1} \bar e-\gamma^{-1} e\right>_{\Gamma_C} \\
& - \gamma^{-1} \|\bar e- e\|_{\Gamma_C}^2.
\end{align*}
For the last term in the right hand side we may add and subtract
$v_h-\bar v_h$ and use the triangle inequality followed by the
interpolation properties of the projection onto piecewise constants
and a trace inequality to obtain
\begin{multline}\label{tracebound}
\gamma^{-1} \|\bar e -e\|_{\Gamma_C}^2 \leq C (\gamma^{-1} \|u - v_h\|_{\Gamma_C}^2 +
  \gamma^{-1} h^{-1} h^2 \|\nabla (v_h - u_h)\|_{h}^2) \\
\leq C (\|u - v_h\|_{1,C}^2+ \gamma^{-1} h^{-1} h^2 \|\nabla e\|_{h}^2)
\end{multline}
As a consequence
\begin{align}\nonumber
\gamma^{-1}\| e\|_{\Gamma_C}^2 +
\left<[P_\gamma u]_+ - [P_\gamma u_h]_+, \gamma^{-1} e\right>_{\Gamma_C}  
\leq {}& \frac14\|\gamma^{-\frac12}(
[P_\gamma u_h]_++ u_h)\|_{\Gamma_C}^2 \\ \nonumber &
+ C (\|u - v_h\|_{1,C}^2+ \gamma^{-1} h^{-1} h^2 \|\nabla e\|_{h}^2)\\ 
& - a_h(e,x_h)+\left<\{\partial_n u\},\jump{x_h} \right>_{\mathcal{F}}.\label{infsup_bound}
\end{align}
Collecting the results of equations \eqref{coercivity1}, \eqref{splita1},
\eqref{1st_step} and \eqref{infsup_bound} and applying the Poincar\'e inequality
\eqref{poincare} leads to
\begin{align}\nonumber
\alpha \left(\frac{1}{2} -   C \frac{h}{\gamma} \right) \|e\|_{1,h}^2 + \frac{1}{2\gamma} \| ([P_\gamma u_h]_++
u_h)\|_{\Gamma_C}^2 \leq {}& - a(e,x_h) +
\left<\partial_n e, v_h-u\right>_{\Gamma_C} \\ \nonumber
& 
- \left<  [P_\gamma u_h]_++ u_h, \partial_n (v_h -u)\right>_{\Gamma_C}\\ \nonumber & +\left<\{\partial_n u\},\jump{v_h-u_h} \right>_{\mathcal{F}_i}\\ \nonumber & +\left<\{\partial_n e\},\jump{x_h} \right>_{\mathcal{F}}\\  &  +C \left(1+\frac{h}{\gamma}\right)\|u - v_h\|_{1,\Gamma_C}^2.\label{eq:bound1}
\end{align}
Observe that $a(u_h,x_h) - \left<\{\partial_n u_h\},\jump{x_h}
\right>_{\mathcal{F}} = 0$ using integration by parts and the
construction of $x_h$. Then, once again by integration by parts, we have
\[
a(e,x_h) - \left<\{\partial_n e\},\jump{x_h} \right>_{\mathcal{F}}
 = (-\Delta u, x_h)_{\Omega_C} \leq \|\Delta u\|_{\Omega_C} \|x_h\|_{\Omega_C},
\]
where $\Omega_C$ is the set of elements with one face on
$\Gamma_C$. Let $h_C>0$ be the largest value such that
$\partial_{h_C} \Omega \cap \Omega_C \ne \emptyset$ and assume that
$h_C \leq t_C$.
Observe that by the construction of $x_h$ and adding and subtracting $v_h$ there holds
\begin{align*}
\|x_h\|_{\Omega_C} \lesssim h^{\frac12} h \gamma^{-1} \|\bar e\|_{\Gamma_C} \lesssim {}& h^{\frac12} h \gamma^{-1} \| e\|_{\Gamma_C} \\
\lesssim {}&
h^{\frac12} (h \gamma^{-1})(\|u-v_h\|_{\Gamma_C} + \| u_h - v_h\|_{\Gamma_C}).
\end{align*}
Let $w_h = u_h - v_h$, then by adding and subtracting $I_\text{cf}
w_h$ and applying the local trace inequality \eqref{trace} and the
standard global trace inequality for functions in $H^1(\Omega)$ we obtain
\begin{align*}
\|w_h\|_{\Gamma_C} \leq {}& \|w_h - I_\text{cf} w_h\|_{\Gamma_C} +
\|I_\text{cf} w_h\|_{\Gamma_C}\\
\lesssim {}&  h^{-\frac12} \|w_h - I_\text{cf} w_h\|_{h}+ 
\|w_h - I_\text{cf} w_h\|_{1,h} + \|w_h\|_{1,h}.
\end{align*}
Applying the discrete interpolation estimate \eqref{discrete_interp},
we then have
\[
\|w_h\|_{\Gamma_C} \lesssim \| w_h\|_{1,h}
\]
from which it follows that
\[
(h \gamma^{-1})  \| u_h - v_h\|_{\Gamma_C} \lesssim (h \gamma^{-1}) (\|e\|_{1,h}+\|u-v_h\|_{1,h}).
\]
For the factor $\|\Delta u\|_{\Omega_C}$ we use \eqref{eq:tracebound} to obtain the bound
\[
\|\Delta u\|_{\Omega_C} \lesssim h^{\frac12}\sup_{0\leq t \leq h_C}
\|\Delta u\|_{\partial_{t} \Omega} \leq h^{\frac12}  \|f\|_{L^2_\infty(\Omega)}.
\]
It follows that 
\begin{equation}\label{eq:bound2}
a_h(e,x_h) - \left<\{\partial_n e\},\jump{x_h}
\right>_{\mathcal{F}} \lesssim h \|f\|_{L^2_\infty(\Omega)}(h \gamma^{-1}) (\|e\|_{1,h}+\|u-v_h\|_{1,h}).
\end{equation}
For the remaining terms of \eqref{eq:bound1} we have by first adding and
subtracting $v_h$ and using the mean
value property of the space $V_h$ and then applying the
Cauchy--Schwarz inequality followed by the arithmetic--geometric inequality,
\begin{equation}
\begin{array}{c}
\left<\partial_n e, v_h-u\right>_{\Gamma_C} 
-\left<  [P_\gamma u_h]_++ u_h, \partial_n (v_h
  -u)\right>_{\Gamma_C} +\left<\{\partial_n u\},\jump{v_h-u_h}
\right>_{\mathcal{F}} \\[3mm]
=\left<\partial_n (u-v_h), v_h-u\right>_{\Gamma_C} + \left<\partial_n (v_h-u_h), v_h-u\right>_{\Gamma_C}\\
- \left<  [P_\gamma u_h]_++ u_h, \partial_n (v_h 
  -u)\right>_{\Gamma_C}\\[3mm]
+\left<\{\partial_n (u - v_h)\},\jump{v_h -u_h}
\right>_{\mathcal{F}_i} \\[3mm]
\leq C \varepsilon^{-1} \|u -v_h\|_{1,C}^2  + \gamma \|\partial_n (u -v_h)\|_{\mathcal{F}_i}^2
+ \frac14 \gamma^{-1}\|[P_\gamma
u_h]_++ u_h\|_{\Gamma_C}^2 \\[3mm]
+ \varepsilon (\gamma
\|\partial_n (v_h-u_h) \|^2_{\Gamma_C}+\gamma^{-1} \|\jump{v_h-u_h}\|_{\mathcal{F}_i}^2).
\end{array}
\end{equation}
Using the zero average property of the nonconforming space,
elementwise trace inequalities and a triangular inequality we obtain
\begin{align*}
\gamma \|\partial_n (v_h-u_h)
\|^2_{\Gamma_C}+\gamma^{-1} \|\jump{v_h-u_h}\|_{\mathcal{F}_i}^2 \leq {}& C_{\gamma_0}
\|v_h-u_h\|^2_{1,h} \\
\leq {}& 2 C_{\gamma_0} (\|e\|^2_{1,h}+\|v_h-u\|^2_{1,h}).
\end{align*}
Observe that $C_{\gamma_0}$ is constant for $\gamma_0 = \gamma/h$ fixed, but it
can not be made small by choosing $\gamma_0$ large (ore small). Instead we choose 
$\varepsilon< \alpha/(16 C_{\gamma_0})$ to obtain the bound
\begin{multline}\label{eq:bound3}
\left<\partial_n e, v_h-u\right>_{\Gamma_C} 
- \left<  [P_\gamma u_h]_++ u_h, \partial_n (v_h
  -u)\right>_{\Gamma_C} +\left<\{\partial_n u\},\jump{v_h-u_h}
\right>_{\mathcal{F}} \\
\leq  C \|u -v_h\|_{1,C}^2  + \gamma \|\partial_n (u -v_h)\|_{\mathcal{F}_i}^2+ \frac{1}{4\gamma} \|([P_\gamma
u_h]_++ u_h)\|_{\Gamma_C}^2+ \frac{\alpha}{8} \|e\|_{1,h}^2.
\end{multline}
Collecting the above bounds \eqref{eq:bound1}, \eqref{eq:bound2} and
\eqref{eq:bound3}, choosing $h \gamma^{-1}$ and $\varepsilon$
small enough (i.e. $\gamma_0$ large enough) we conclude that for all $v_h \in V_h$
\begin{align*}
\alpha^{\frac12} \|e\|_{1,h} + \gamma^{-\frac12}\| [P_\gamma
u_h]_++  u_h\|_{\Gamma_C} 
\lesssim {}& (\|u - v_h\|_{1,C} + h^{\frac12}
\|\partial_n (u -v_h)\|_{\mathcal{F}_i}) \\
& + \frac{h}{\alpha^{\frac12}} \|f\|_{L^2_\infty(\Omega)}.
\end{align*}
\end{proof}
\begin{corollary}
Under the assumptions of Theorem \ref{thm:apriori} there holds
\begin{equation}\label{estimate_2}
\alpha^{\frac12} \|e\|_{1,h} + \gamma^{-\frac12}\| [P_\gamma
u_h]_++  u_h\|_{\Gamma_C} 
\lesssim h^{\frac12+\nu} \|u\|_{H^{\frac32+\nu}(\Omega)}
+ \frac{h}{\alpha^{\frac12}} \|f\|_{L^2_\infty(\Omega)}.
\end{equation}
\end{corollary}
\begin{proof}
This is immediate from the best approximation result of Theorem
\ref{thm:apriori} and the existence of an optimal approximation of $u$
in $V_h$. Since the Crouzeix-Raviart space contains the
$H^1$-conforming space of piecewise affine functions we may take the
standard Lagrange interpolant $i_h u$ for which there holds (see \cite{DS80,CHR15}).
\[
\|u - i_h u\|_{1,C} + h^{\frac12}
\|\partial_n (u -i_h u)\|_{\mathcal{F}_i}\lesssim h^{\frac12+\nu} \|u\|_{H^{\frac32+\nu}(\Omega)}.
\]
\end{proof}
\section{Numerical example}
Here we will consider two examples on the unit square,
$\Omega=[0,1]^2$. We have used the package FreeFEM++ for the
computations \cite{He12}. We let $\Gamma_D=[0,1]\times\{1\}$,
$\Gamma_N=\{0\}\times [0,1] \cup  \{1\}\times [0,1]$ and $\Gamma_C
=[0,1]\times\{0\}$. In all cases we use a fixed point iteration to
compute the solution and we iterate until the relative $H^1$-error of
the increment if smaller than $10^{-5}$.

In the graphics below the $H^1$-error is marked with squares,
the $L^2$-error with circles and finally the residual quantity $\|u_h +
[P_\gamma(u_h)]_+\|_{\Gamma_C}$ by triangles. Dotted lines are
reference lines with slopes $O(h)$ (upper)  and $O(h^2)$ (lower).

\subsection{Problem with known solution}\label{sec:known}
We first consider an example where the exact solution is known,
\[
u(x,y) := -\cos(\pi/2\,y)\,\sin ^2 (\pi\,x)
\]
with the right hand side
\[
f=\frac{\pi^2}{4}\cos(\pi/2 y)\sin ^2 (\pi x)-2 \pi^2\cos(\pi/2\, y)\cos(2\pi\,x).
\]
Observe that this actually is a solution to a linear Neumann problem,
but we can still use it as a solution to the nonlinear problem. The
contact takes place in the set $\{0,1\}$. We
solve it on a sequence of Union Jack style meshes (see the left plot
of Figure \ref{mesh_solution} for an example) with $h/\sqrt{2} \in \{ 2^{-(i+4)}
\}_{i=0}^{3}$. The result is presented in the left
graphic of Figure \ref{fig:exa2}. As expected we observe first order
convergence of the relative $H^1$-error and second order
convergence of the relative  $L^2$-error. As expected the residual quantity has
a convergence close to $\mathcal{O}(h^{\frac32})$. 
\subsection{Problem with unknown solution}\label{sec:unknown}
Here we propose the problem obtained by setting
\begin{equation}\label{sourceterm}
f= ( 2 \pi N)^2 \cos(2 \pi N x), N \in \{3,5\}.
\end{equation}
We solve the problem on a mesh with $h=2\sqrt{2} \cdot 10^{-3}$ (a
$500\times 500$ mesh), using the
nonsymmetric Nitsche method from \cite{CHR15} and piecewise quadratic
conforming approxmation to obtain a reference
solution. We report the contour lines of the solution for $N=5$ in the
right plot of Figure \ref{mesh_solution}. Then we solve the problem for $h/\sqrt{2} \in \{ 2^{-(i+4)}
\}_{i=0}^{4}$ and compute the same quantities as in
the previous case. The convergences are reported in the right
graphic of Figure \ref{fig:exa2}. The cases $N=3$ and $N=5$ are
                                 distinguished by the use of white
                                 and black markers respectively,
                                 similar convergence orders were
                                 observed in both cases. 
First order convergence is observed for the error in the $H^1$-norm
and second order convergence in the $L^2$--error. As before the
convergence of $\|u_h +
[P_\gamma(u_h)]_+\|_{\Gamma_C}$ is approximately
$\mathcal{O}(h^{\frac32})$. 
\begin{figure}
\begin{center}
\hspace{-1cm}
\includegraphics[angle=90,height=6cm]{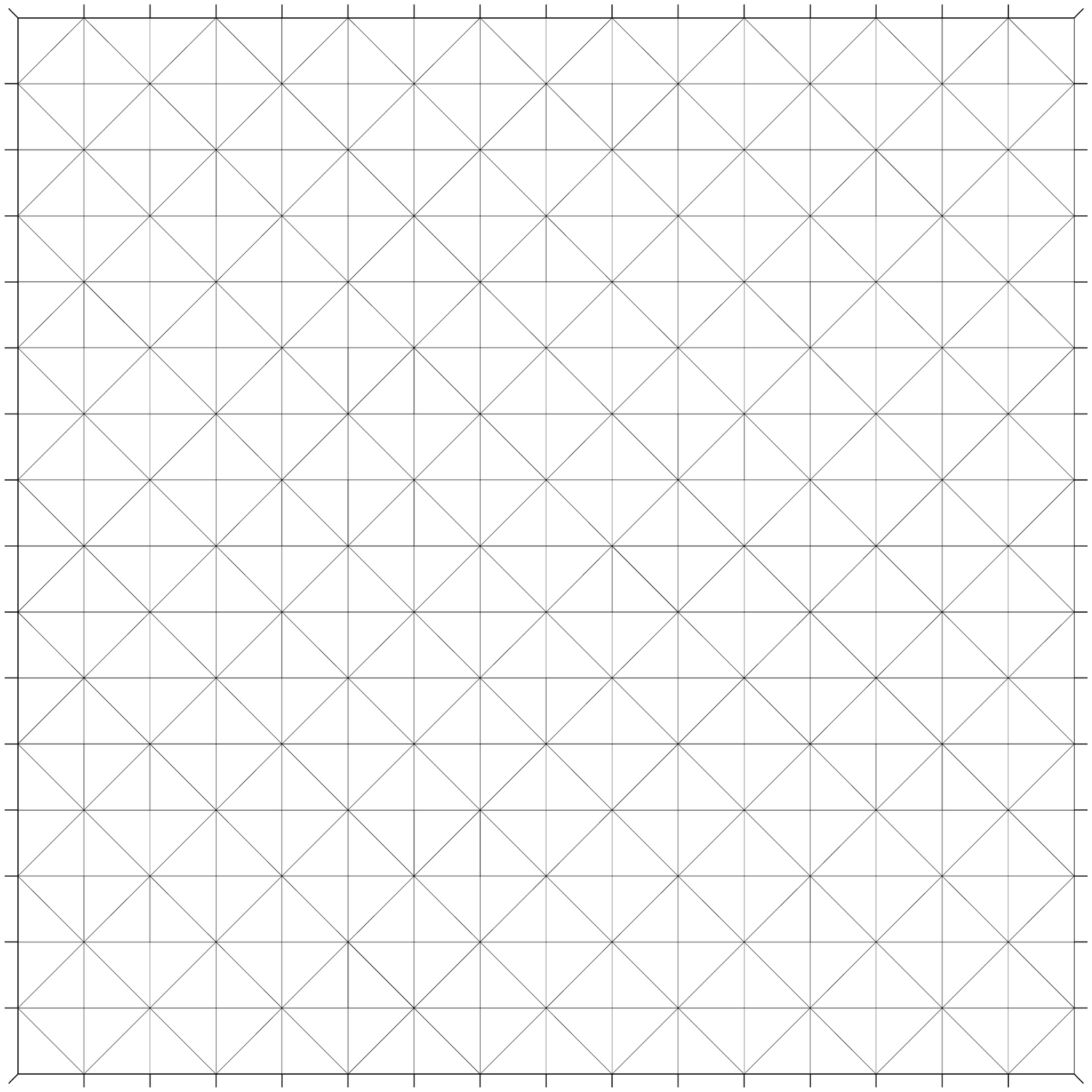}\hspace{-2cm}\includegraphics[trim=0
60 0 60,clip,angle=90,height=6cm
]{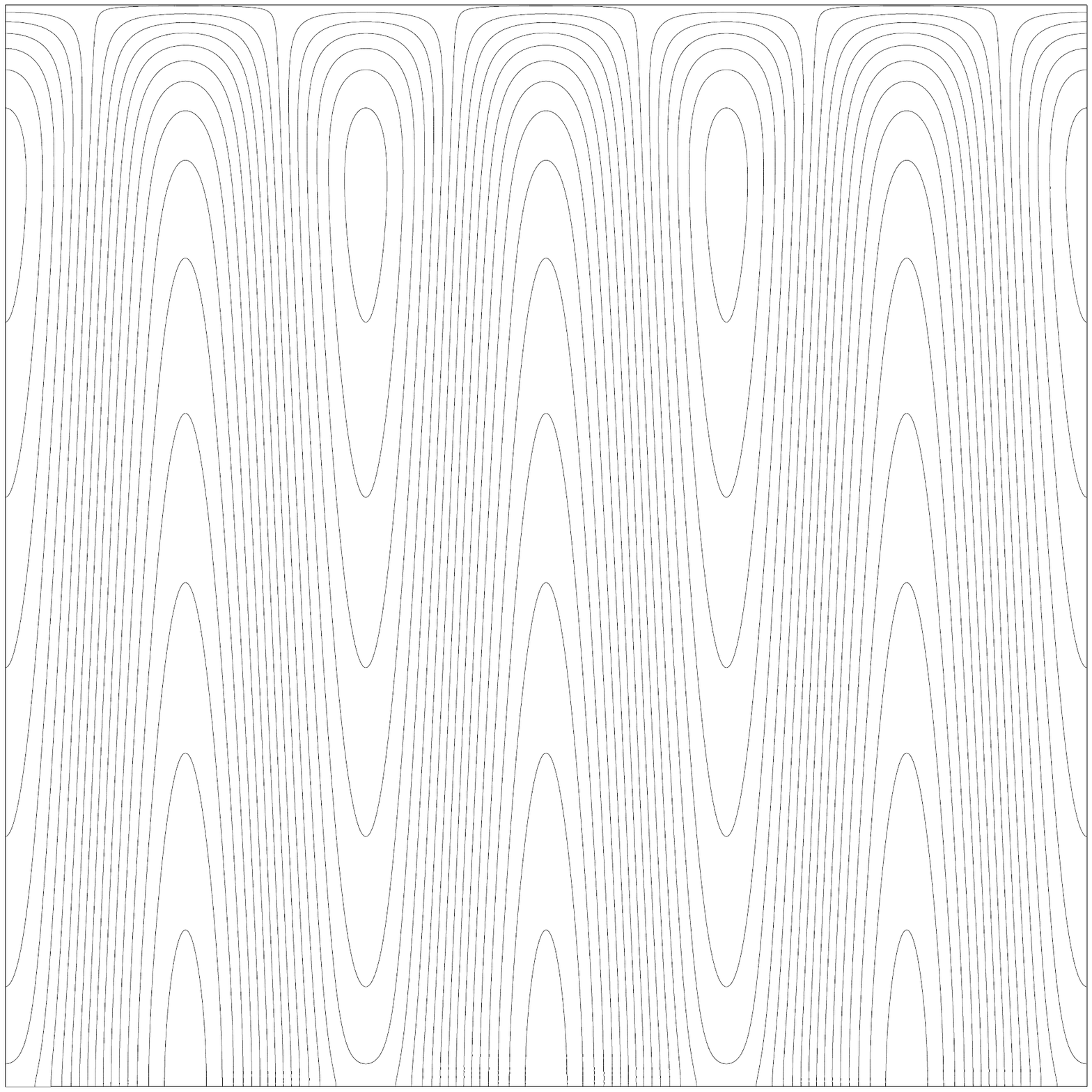}
\caption{Left: example of a computational mesh. Right: the fine mesh solution
  using \eqref{sourceterm} with $N=5$.}\label{mesh_solution}
\end{center}
\end{figure}
\begin{figure}
\begin{center}
\includegraphics[height=6cm]{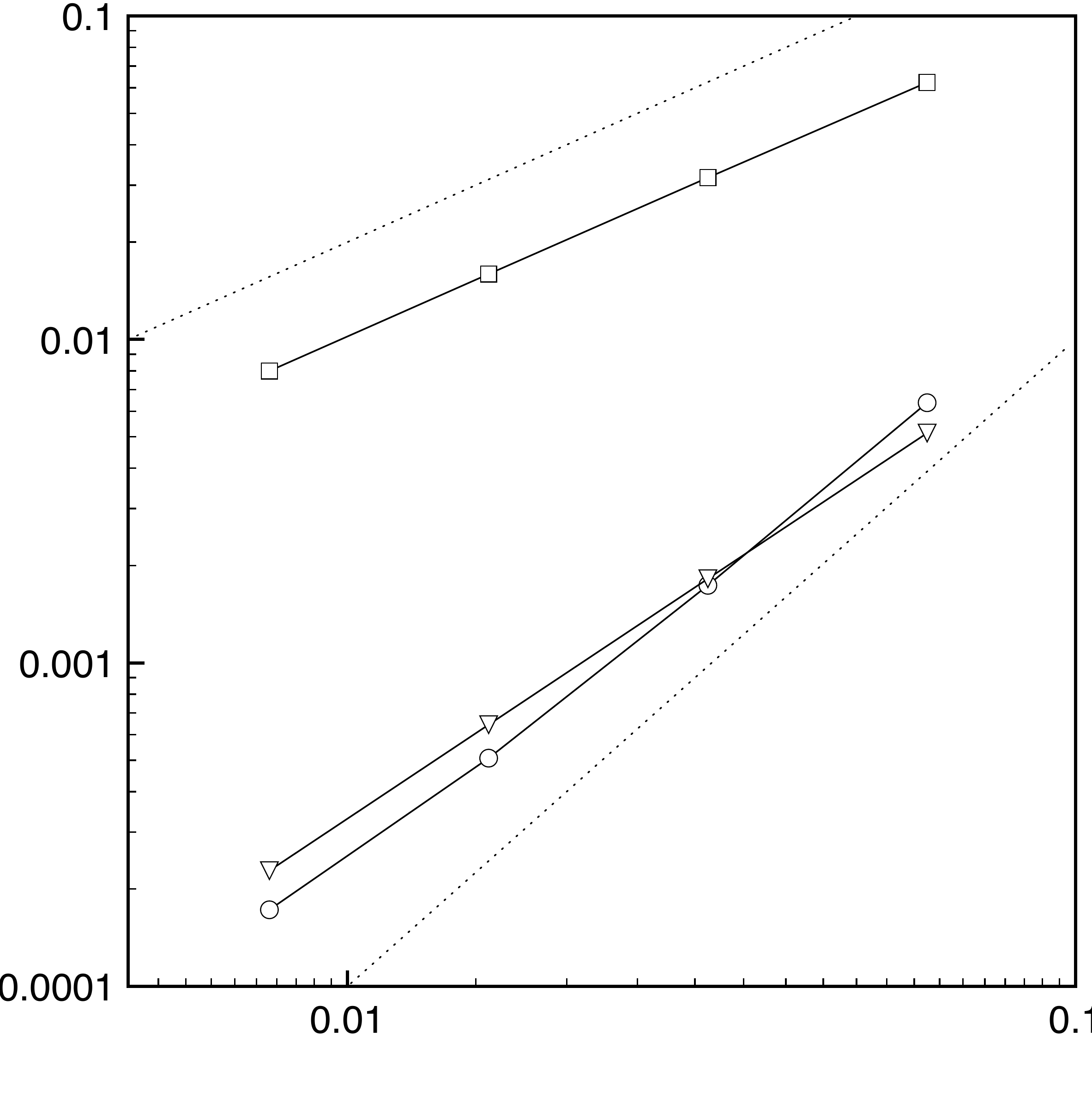}\hspace{0.5cm}\includegraphics[height=6cm]{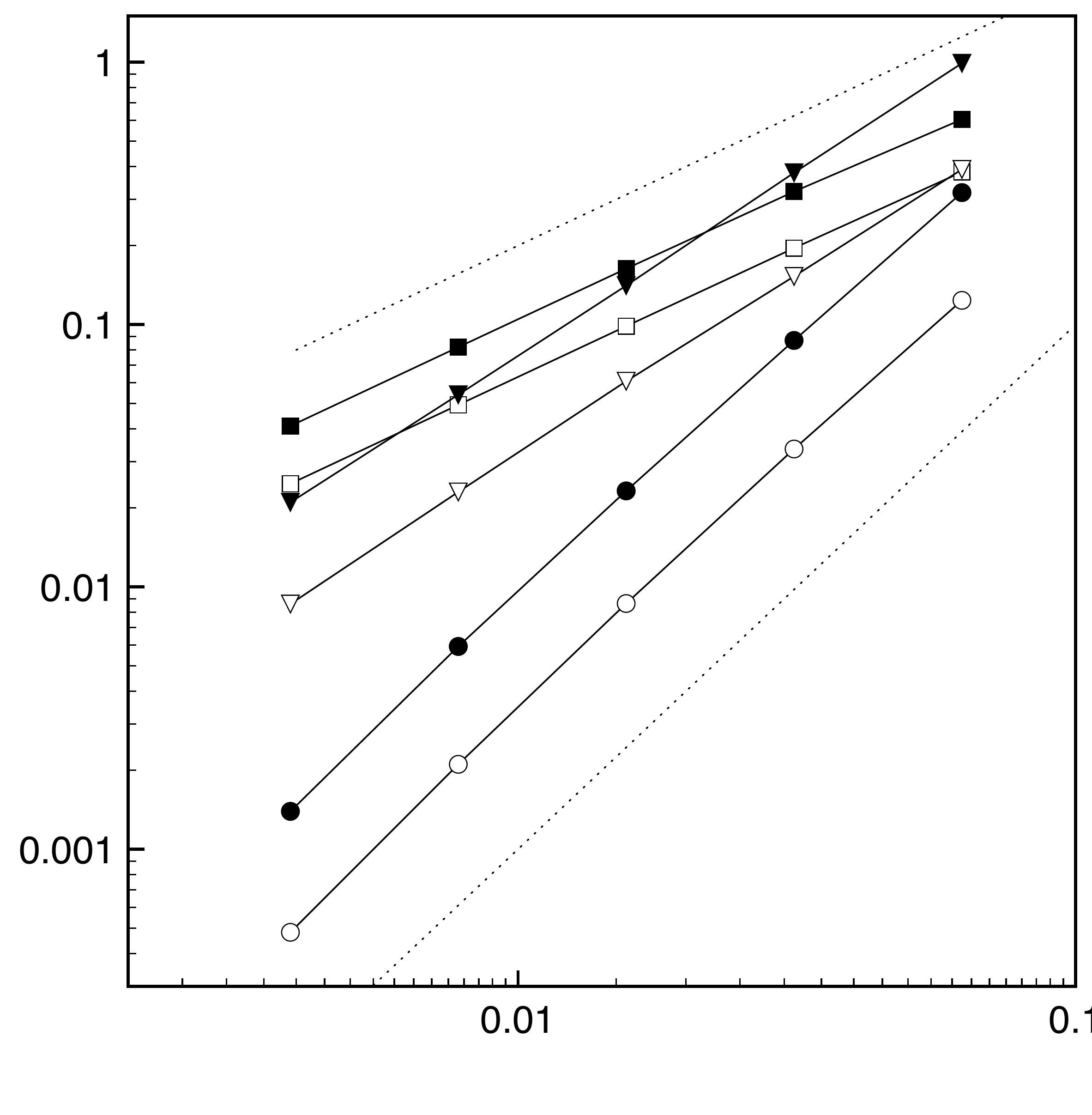}
\caption{Convergence plots of the two numerical examples. Left the
  problem from Section \ref{sec:known}. Right from Section \ref{sec:unknown}. Dotted lines
are reference curves. Upper $O(h)$, lower $O(h^2)$. Square markers -
$H^1$-error; circle markers - $L^2$--error; triangle markers -
satisfaction of the contact condition,
$\|u_h+[P_\gamma(u_h)]_+\|_{\Gamma_C}$. In the right plot, white
markes indicate $N=3$ and black markers $N=5$.}\label{fig:exa2}
\end{center}
\end{figure}
\section{Conclusion}
We have proved that the nonsymmetric Nitsche method of \cite{Bu12} may
be applied in the framework of \cite{CH13b,CHR15} for the approximation
of unilateral contact problems. An optimal error estimate for a method
using a nonconforming finite element space was derived
combining tools from the inf-sup analysis of \cite{Bu12}
with the monotonicity argument of \cite{CH13b,CHR15}. The theoretical
results were illustrated in two numerical examples. Herein we only
considered the simplified case of the Signorini problem based on
Poisson's equation, but the extention to elasticity may be feasible
using the results from \cite{BB16}. Another natural question is if the
above analysis can be extended to the case of standard conforming
elements. The difficulty here is to handle the non-local character of
the function necessary for the stability argument, adding a layer of
terms that must be estimated. Numerical
experiments not reported here indicate that the conforming method also
performs well.
\bibliographystyle{plain}
\bibliography{contact}

\end{document}